\documentclass[11pt]{amsart}
\newtheorem{theorem}{Theorem}[section]

\newtheorem{proposition}[theorem]{Proposition}

\newtheorem{definition}{Definition}[section]
\newtheorem{remark}[theorem]{Remark}

\numberwithin{equation}{section}

\title{On the Toda systems of VHS type}

\author{Chen-Yu Chi}
\address{Department of Mathematics and Taida Institute for Mathematical Sciences, National Taiwan University, Taipei, Taiwan.}
\email{chi@tims.ntu.edu.tw, chi1@ntu.edu.tw}

\begin{document}
\maketitle

\begin{abstract}
We consider the Toda systems of VHS type with singular sources and provide a criterion for the existence of solutions with prescribed asymptotic behaviour near singularities. We also prove the uniqueness of solution. Our approach uses Simpson's theory of constructing Higgs-Hermitian-Yang-Mills metrics from stability. 
\end{abstract}

\normalsize

\section{Introduction}\label{int}
Let $M$ be a compact Riemann surface and $g=ds^2$ be a smooth riemannian metric on $M$, and denote the Laplacian associated to $g$ by $\Delta_g$ and the Gaussian curvature of $g$ by $K_g$. For $\epsilon=\pm 1$, we consider systems of partial differential equations of the following form

\begin{eqnarray}\label{toda}
\epsilon\begin{pmatrix}
\frac{1}{4}\Delta_gu_1-\frac{K_g}{2}\\
\\
\frac{1}{4}\Delta_gu_2-\frac{K_g}{2}\\
\\
\vdots\\
\\
\frac{1}{4}\Delta_gu_n-\frac{K_g}{2}
\end{pmatrix}
=
\begin{pmatrix}
2&-1&&\\
\\
-1&2&&\\
\\
&&\ddots&-1\\
\\
&&-1&2
\end{pmatrix}
\begin{pmatrix}
e^{u_1}\\
\\
e^{u_2}\\
\\
\vdots\\
\\
e^{u_n}
\end{pmatrix}
\end{eqnarray}
\\
on $M$ with finitely many points removed. We call it a Toda system of VHS type or hermitian type according to $\epsilon=1$ or $-1$. Here VHS stands for (polarized complex) variation of Hodge structure. The reason of the name will be clear later. Let $\mu=(\mu_1,\dots,\mu_n)$ be an $n$-tuple of functions defined on $M$ each of which vanishes at all but finitely many points. The $\mu_j$s are called singular strengths and $S:=\{p\in M:\mu_j(p)\neq 0\text{ for some }j\}$ the set of punctures. 
\begin{definition}
A $\mu$-admissible solution to (\ref{toda}) is an $n$-tuple of real-valued smooth functions $(u_1,\dots,u_n)$ on $M\setminus S$ satisfying (\ref{toda}) and behaving near the punctures in the following manner: for each $p\in S$ there exists a sufficiently small coordinate chart $(U,z)$ centered at $p$ and smooth {\it bounded} functions $v_j$ on $U$ such that
$$u_j=2\mu_j(p){\rm log}|z|+v_j, $$
on $U\setminus\{p\}$, $j=1,\dots,n$.
\end{definition}

The Toda systems we consider here are usually called type $A$, manifesting its relation to $A_{n+1}$. One can consider more general types of Toda systems by replacing the Cartan matrix of type $A$ by those of other types. For the case of smooth solutions on $M$, i.e. $\mu=(0,\dots,0)$, there have appeared many studies relating Toda systems with harmonic maps and Higgs bundles, for example, \cite{baraglia} and \cite{jost-wang}. The case with nontrivial singular sources is more involved and is the subject of the current paper. 

Despite the formal similarity, these two types of Toda systems are quite different in nature, both analytically and geometrically. The VHS type is related to the notion of stability and the hermitian type is more related to the consideration of harmonic maps. In this paper, after developing a general geometric formalism of these systems, which is similar for both types, we will focus on Toda systems of VHS type and treat those of hermitian type in another paper. 

Our main result is the following theorem (Theorem \ref{main thm}).\\
\\
{\bf Theorem.} {\it For any assignment of singular strengths $\mu=(\mu_1,\dots,\mu_n)$ there exists at most one $\mu$-admissible solution to the Toda system of VHS type. There exists a $\mu$-admissible solution $(u_1,\dots,u_n)$ to the Toda system of VHS type if and only if
$$d_{n-l+1}+\cdots+d_n <l(n-l+1)({\rm genus}(M)-1),$$
$l=1,\dots,n$, where 
$$d_k:=\sum_{p\in M}\left(-\frac{1}{n+1}\sum_{j=1}^n(n-j)\mu_j(p)+\sum_{j=1}^k\mu_j(p)\right),$$
$k=1,\dots,n$.}\\

The arrangement of this paper is as follows. In Section \ref{Cvhs}, we introduce the notion of complex pre-VHS and diagonality and show that every Toda system in the above sense corresponds to a special type of complex variation of Hodge structure over the punctured Riemann surface $X=M\setminus S$ whose underlying bundle is a canonically chosen smooth hermitian vector bundle $(V,h)$. If the Gauss-Manin connection preserves the hermitian form $h(C\cdot,\cdot)$ ($C$ being the Weil operator), then the Toda system is of VHS type (the traditional polarization); if instead it preserves the hermitian metric $h$, the Toda system is of hermitian type. In Section \ref{Higgs}, we recall the notion of Higgs bundle and introduce a Higgs bundles $E$ on $X$ related to a system of VHS type. It is the restriction of a natural vector bundle $\widetilde V$ on $M$.

In Section \ref{Higgs}, we establish a correspondence between $n$-tuples of functions on $M\setminus S$ which behave suitably near the punctures and some type of hermitian metrics on $E$ with corresponding asymptotic property. The procedure mimics that of establishing a correspondence between complex variations of Hodge structure and system of Hodge bundles as in \cite{simpson}, simply dropping the flatness requirement on curvature. Under this correspondence, solutions to Toda systems correspond to flat metrics. 

Our main tool in getting the criterion for existence of solutions is the theory of getting Higgs-Hermitian-Yang-Mills metrics (which are flat in our case) from stability, developed by Hitchin \cite{hitchin} and Simpson \cite{simpson}. As mentioned earlier, the presence of singularities is the main difficulty to be overwhelmed. In order to deal with singularities, we have to use the results of \cite{simpson} for quasiprojective curves.\\
\\
{\bf Acknowledgements} \\

I am indebted to Professors Chang-Shou Lin and Chin-Lung Wang for their many valuable suggestions. This work was partially supported by National Science Council of Taiwan.

\section{Complex variation of Hodge structure}\label{Cvhs}
In this section we relate Toda systems with flat connections. The formalism has appeared in studies of Toda systems on a region of $\mathbf R^2$, for example, in \cite{jost-wang}. The vector bundles underlying the flat connections in these classical situations are mainly trivial bundles. We propose a simple globalization as the preparation for further development.
  
Let $X$ be a complex manifold and denote the sheaf of germs of smooth functions by $\mathcal A_X=\mathcal A_X^0$. 
\begin{definition}\label{cvhs}
$(1)$ A complex pre-variation of Hodge structure (complex pre-VHS for short) $(\{V^{r,s}\},\nabla)$ of weight an integer $n$ on $X$ consists of 
\begin{itemize}
\item[(i)] smooth complex vector bundles $V^{r,s}$ over $X$, $r,s\in\mathbf Z$, $r+s=n$, with $V^{r,s}=0$ for all but finitely many $(r,s)$ and
\item[(ii)] a connection $\nabla:\mathcal A_X^0(V)\to\mathcal A_X^1(V)$ on $V:=\oplus_{p+q=n} V^{r,s}$  
which only has components of {\it total} degree $(1,0)$ and $(0,1)$. In other words, $\nabla=\oplus_{r,s}\nabla^{p,q}$ with 
$$\nabla^{r,s}:\mathcal A_X^0(V^{r,s})\rightarrow\mathcal A_X^{1,0}(V^{r-1,s+1})\oplus \mathcal A_X^{1,0}(V^{r,s})\oplus \mathcal A_X^{0,1}(V^{r,s})\oplus\mathcal A_X^{0,1}(V^{r+1,s-1})$$
for each pair $(r,s)$. If we write $\nabla=\theta+\nabla^{\rm Hodge}+\theta'$ where
$$\nabla^{\rm Hodge}=\oplus_{r+s=n}\{\nabla^{{\rm Hodge},r,s}:\mathcal A_X^0(V^{r,s})\longrightarrow\mathcal A_X^{1,0}(V^{r,s})\oplus\mathcal A_X^{0,1}(V^{r,s})\},$$
$$\theta:=\oplus_{r+s=n}\{\theta^{r,s}:\mathcal A_X^0(V^{r+1,s-1})\longrightarrow\mathcal A_X^{1,0}(V^{r,s})\},$$ 
and 
$$\theta':=\oplus_{r+s=n}\{\theta'^{r,s}\mathcal A_X^0(V^{r-1,s+1})\longrightarrow\mathcal A_X^{0,1}(V^{r,s})\},$$
then it is clear that
$\nabla^{\rm Hodge}$ is also a connection and $\theta$ and $\theta'$ are $\mathcal A_X$-linear. $\theta$ is called the pre-Hodge field of the complex pre-VHS $(\{V^{r,s}\},\nabla)$.
\end{itemize}
$(2)$
A complex pre-VHS $(\{V^{r,s}\},\nabla)$ is a complex VHS if furthermore the curvature of $\nabla$ is $0$. If this is the case, $\theta$ is called the Hodge field of the complex VHS.
\end{definition}
\begin{definition}\label{pol}
Let $(\{V^{r,s}\},\nabla)$ be a complex pre-VHS of weight $n$. A Hodge-polarization (resp. hermitian-polarization) consists of $\{h^{r,s}\}$ where $h^{r,s}$ is a smooth hermitian metric on $V^{r,s}$ for each $(r,s)$ such that if $h:=\oplus_{r+s=n} h^{r,s}$ then $h(C\cdot, \cdot)$ (resp. $h(\cdot, \cdot)$) is preserved by $\nabla$, where $C$ is the Weil operator defined by 
$$C|_{V^{r,s}}:=i^{r-s}{\rm id}_{V^{r,s}}=i^n(-1)^s{\rm id}_{V^{r,s}}.$$   
More precisely, this means that 
$$X\langle\sigma ,\tau\rangle =\langle\nabla_X\sigma ,\tau\rangle +\langle\sigma ,\nabla_{\overline X}\tau\rangle,$$ for all $X\in T_xX$, $\sigma,\tau\in \mathcal A_X^0(V)_x$, and $\langle\cdot,\cdot\rangle=h(C\cdot,\cdot)$ (resp. $h(\cdot,\cdot)$). If this is the case, we say that $(\{V^{r,s}\},\nabla)$ is Hodge(resp. hermitian)-polarized by $\{h^{r,s}\}$ (or $h$ for short). It is not hard to see that $\{h^{r,s}\}$ is a Hodge-polarization (resp. hermitian-polarization) if and only if $\nabla^{\rm Hodge}$ preserves $h$ and $\theta'=\theta^*$ (resp. $-\theta^*$), where 
$$\theta^*:\mathcal A_X^0(V)\longrightarrow \mathcal A_X^{0,1}(V)$$ is the adjoint of $\theta$ with respect to $h$.
\end{definition}
Now we specialize to the situation related to Toda systems. Let $M$ be a Riemann surface and $g$ a riemannian metric on $M$. Let $S$ be any closed subset of $M$ and $X:=M\setminus S$. 
Let $K_M$ be the canonical line bundle of $M$. We have a hermitian metric on $K_M$ naturally associated to $g$: if locally one writes $g=\varphi\cdot\overline\varphi$ for some nonvanishing $(1,0)$-form $\varphi$, then $\varphi$ and is a unitary frame of $K_M$. Fix a smooth complex line bundle $L$ such that 
\begin{equation}\label{1}
L^{\otimes (n+1)}=K_M^{\otimes\frac{n(n+1)}{2}}.
\end{equation} 
We define 
\begin{equation}\label{2}
\widetilde V^{n-k,k}:=L\otimes K_M^{\otimes -k}, k=0,\dots,n.
\end{equation} 
and equip $\widetilde V^{n-k,k}$ with the hermitian metrics $\tilde h^{n-k,k}$ canonically associated to that of $K_M$. Let $V^{n-k,k}:=\widetilde V^{n-k,k}|_X$ and $h^{n-k,k}:=\tilde h^{n-k,k}|_V$. Suppose $\nabla$ is a connection on $V$ such that
\begin{itemize}
\item[(a)] ($\{V^{n-k,k}\}, \nabla)$ is a complex pre-VHS which is Hodge(resp. hermitian)-polarized by $\{h^{n-k,k}\}$ and 
\item[(b)] all the corresponding $\theta^{n-k,k}$ are non vanishing.
\end{itemize}

\begin{remark}\label{theta-b}
For each $k$, $\theta^{n-k,k}$ can be viewed as a function $b_k$ on $X$, since 
$$K_M\otimes{\rm Hom}\left(L\otimes K_M^{\otimes -k+1}, L\otimes K_M^{\otimes -k}\right)$$ 
is trivial. Therefore, when talking about the pre-Hodge field of a complex pre-VHS of the form in (\ref{2}) (which is the only kind of complex pre-VHS interesting in the following), we will use $(b_1,\dots,b_n)$ and $\theta$ interchangeably. 
\end{remark}
Let $\varphi$ and $e_L$ be unitary local frames of $K_M$ and $L$ respectively such that 
$e_L^{\otimes(n+1)}$ corresponds to $\varphi^{\otimes\frac{n(n+1)}{2}}$.
Then 
$$e_k:=e_L\otimes \varphi^{\otimes-k}$$ 
is a unitary local fram of $\widetilde V^{n-k,k}$ for each $k$, with respect to which we can express the connection form of $\nabla^{\rm Hodge}$ as a skew-hermitian trace-free diagonal matrix of $1$-forms
$$A^{\rm Hodge}=\begin{pmatrix}
A_0&&&&\\
&\ddots&&&\\
&&A_{k}&&\\
&&&\ddots&\\
&&&&A_n
\end{pmatrix}$$
and $\theta$ as a matrix of $(1,0)$-forms
$$B=\begin{pmatrix}
0&&&\\
B_1&0&&\\
&\ddots&\ddots&\\
&&B_n&0
\end{pmatrix},$$
where 
\begin{equation}\label{3}
B_k=b_k\varphi,\ k=1,\dots,n.
\end{equation}
$\theta'=\epsilon\theta^*$, $\epsilon=\pm 1$ according to the type of polarization, and hence the connection form of $\nabla$ is
\begin{equation}\label{4}
\omega=
\begin{pmatrix}
A_0&\epsilon\overline{B}_1&&&\\
B_1&A_1&\epsilon\overline{B}_2&&\\
&\ddots&\ddots&\ddots&\\
&&B_{n-1}&A_{n-1}&\epsilon\overline{B}_n\\
&&&B_n&A_n
\end{pmatrix}.
\end{equation}
Then the curvature $F_\nabla=d\omega+\omega\wedge\omega$, whose $(k,l)$-entry, $k\leq l$, is
$$\left\{
\begin{array}{cl}
dA_k+\epsilon (B_k\wedge\overline{B}_k-B_{k+1}\wedge\overline{B}_{k+1}),&k=l=0,1,\dots,n;\\
\\
dB_k+(A_{k}-A_{k-1})\wedge B_k,&k=l+1=1,\dots,n;\\
\\
0,&\text{otherwise},
\end{array}
\right.$$
with the convention that $B_0=0=B_{n+1}$.
Then the complex pre-VHS $(\{V^{n-k,k}\}, \nabla)$ is a complex VHS if and only if $F_\nabla =0$, i.e. locally 
\begin{equation}\label{5}
dA_k=-\epsilon (B_k\wedge\overline{B}_k-B_{k+1}\wedge\overline{B}_{k+1}),
\end{equation}
$k=0,1,\dots,n$
and 
\begin{equation}\label{6}
dB_k=(A_{k-1}-A_{k})\wedge B_k,
\end{equation}
$k=1,\dots,n$.
Note that $d\varphi=-i\rho\wedge\varphi$ for a unique real $1$-form $\rho$ and $$d\rho=-\frac{i}{2}K_g\ \varphi\wedge\overline\varphi.$$ (\ref{6}) then becomes
\begin{equation}\label{7}
\overline\partial b_k-ib_k\rho^{0,1}=(A_{k-1}^{0,1}-A_{k}^{0,1})b_k,
\end{equation}
$k=1,\dots, n.$
Since $A^{\rm Hodge}$ is skew-hermitian,
\begin{equation}\label{8}
-\partial \overline{b_k}-i\overline{b_k}\rho^{1,0}=(A_{k-1}^{1,0}-A_{k}^{1,0})\overline{b_k}, 
\end{equation}
$k=1,\dots, n$. By the nonvanishing assumption on $b_k$, \ref{7} (or \ref{8}) is equivalent to
\begin{equation}\label{9}
A_{k-1}-A_{k}=
\frac{\overline\partial b_k}{b_k}
-\frac{\partial\overline b_k}{\overline b_k}-i\rho, 
\end{equation}
$k=1,\dots, n$. On the other hand, by (\ref{5}) and (\ref{9}), we have
\begin{equation}\label{10}
\partial\overline\partial\ln|b_k|^2-id\rho=
\epsilon(-|b_{k-1}|^2+2|b_{k}|^2-|b_{k+1}|^2)\varphi\wedge\overline\varphi,
\end{equation}
$k=1,\dots, n$, which is equivalent to 
\begin{equation}\label{11}
\frac{1}{4}\Delta_g\ln|b_k|^2-\frac{K_g}{2}=\epsilon(-|b_{k-1}|^2+2|b_{k}|^2-|b_{k+1}|^2),
\end{equation}
$k=1,\dots,n$, namely, 
\begin{equation}\label{12}
(u_1,\dots,u_n):=\big(\ln|b_1|^2,\dots,\ln|b_n|^2\big)
\end{equation} 
is a solution of the Toda system. Note that 
\begin{equation}\label{13}
A_0+A_1+\dots +A_n=0.
\end{equation}
By (\ref{9}), we have
\begin{equation}\label{14}
A_0=\frac{1}{n+1}\left(\frac{\overline\partial(b_1^nb_2^{n-1}\cdots b_n)}{b_1^nb_2^{n-1}\cdots b_n}-\frac{\partial(\overline b_1^n\overline b_2^{n-1}\cdots\overline b_n)}{\overline b_1^n\overline b_2^{n-1}\cdots \overline b_n}\right)-i\frac{n}{2}\rho.
\end{equation}

Conversely, suppose $b_1,\dots, b_n$ are nonvanishing smooth functions on $X=M\setminus S$ satisfying (\ref{11}) (and hence (\ref{10})). We can define $A_k$ and $B_l$ locally by the formulas in (\ref{9}), (\ref{14}), and (\ref{3}). Then we put them together to form $\omega$ as in (\ref{4}). It is easy to see this local description actually gives a global connection on $V$. By (\ref{9}), such a connection satisfies (\ref{6}).   
Finally, (\ref{11}) and (\ref{13}) together imply (\ref{5}). In summary, we have obtained the following statement. 
\begin{proposition}\label{toda-cvhs1} 
Let $(M,g)$ be a Riemann surface with a smooth riemannian metric, $X$ an open subset of $M$, $L$ a smooth complex line bundle with $L^{\otimes (n+1)}=K_M^{\otimes\frac{n(n+1)}{2}}$, $V^{n-k,k}:=(L\otimes K_M^{\otimes -k})|_X$, and $h^{n-k,k}$ the metric canonically associated to $g$, $s=0,\dots,n$. There exists a one-one correspondence between solutions $(b_1,\dots, b_n)$ to (\ref{11}) with $\epsilon= 1$ (resp. $-1$) on $X$ and connections $\nabla$ on $V=\oplus_k V^{n-k,k}$ such that $(\{V^{n-k,k}\},\nabla)$ is a complex variation of Hodge structure Hodge (resp. hermitian)-polarized by $\{h^{n-k,k}\}$ with all $\theta^{n-k,k}$ nonvanishing.
\end{proposition}
It is obvious that two solutions to (\ref{11}) give the same solution (\ref{12}) if their corresponding components differ from each other by a phase, that is a smooth function whose values are complex numbers of unit length. We can rephrase the above proposition as follows.
\begin{proposition}\label{toda-cvhs2}
Under the assumption of the previous proposition, there exists a one-one correspondence between solutions $(u_1,\dots, u_n)$ to (\ref{toda}) with $\epsilon= 1$ (resp. $-1$) on $X$ and connections $\nabla$ on $V=\oplus_k V^{n-k,k}$ such that $(\{V^{n-k,k}\},\nabla)$ is a complex variation of Hodge structure Hodge(resp. hermitian)-polarized by $\{h^{n-k,k}\}$ with all Higgs components $b_k$ positive.
\end{proposition}
We introduce the following auxiliary notion.
\begin{definition}\label{od}
A complex pre-VHS $(\{\widetilde V^{r,s}\},\nabla)$ over $M$ is called diagonal if the curvature of $\nabla$ is diagonally valued, i.e.
$$F_{\nabla}(Y_1,Y_2)(V^{r,s}_x)\subset V^{r,s}_x\text{, for all }(r,s), Y_1,Y_2\in T_xM, x\in M.$$
\end{definition}
\begin{proposition}\label{d-pre-vhs}
For a Riemann surface with smooth metric $(M,g)$, if we have vector bundles $\{V^{n-k,k}\}$ on $X=M\setminus S$ as in (\ref{2}) with the hermitian metric $\{h^{n-k,k}\}$ canonically associated to $g$ and $\mathcal A_X$-linear maps
$$\theta^{n-k,k}:\mathcal A_X^0(V^{n-k+1,k-1})\longrightarrow\mathcal A_X^{1,0}(V^{n-k,k}),$$
there exists a unique connection $\nabla$ on $V$ making $(\{V^{n-k,k}\}, \nabla)$ a diagonal complex pre-VHS which is polarized by $\{h^{n-k,k}\}$ and whose pre-Hodge field is given by $\{\theta^{n-k,k}\}$.
\end{proposition}
\begin{proof}
By checking the discussion above carefully, the local condition for the diagonality in this situation is precisely (\ref{6}), which implies (\ref{9}), and hence (\ref{14}) by (\ref{13}). Therefore every $A_k$ is determined by all $b_s$(i.e. by $\theta$). It is also easy to see that the local descriptions patch together well to yield a global connection, which is clearly diagonal. 
\end{proof}
\begin{remark} In view of this proposition, we can view the Toda system, or more precisely (\ref{11}), as the condition on a ``potential pre-Hodge field" $\theta=(b_1,\dots,b_n)$ such that its uniquely associated diagonal complex pre-VHS is a complex VHS.
\end{remark}
We could have had the entire discussion above on a Riemann surface with metric $(X,g)$ without mentioning any ambient manifold $M$ and the complement $S$. Later we will come back to the situation in Section \ref{int} that $(M,g)$ be a compact Riemann surface with a smooth metric and 
$$S:=\{p\in M:\mu_j(p)\neq 0\text{ for some }j\}$$ 
is the set of punctures of an assignment of singular strengths $\mu_j$, and this is the reason we chose to state the results in the above manner.
\section{The associated Higgs bundles of an assignment of singular strehgths}\label{Higgs}
We recall the definition in \cite{simpson}. Let $X$ be a complex manifold.
\begin{definition}\label{higgs}
A Higgs bundle $(E,\Phi)$ consists of
\begin{itemize}
\item[(i)] a holomorphic vector bundle $E$ over $M$ and
\item[(ii)] a holomorphic ${\rm End}(E)$-valued $(1,0)$-form $\Phi$, called the Higgs field. 
\end{itemize}
\end{definition}
Suppose $H$ is a hermitian metric on $E$. We will denote the Chern connection of $E$ associated to $H$ by $\nabla^H$ and the adjoint of $\Phi$ with respect to $H$ by $\Phi^*$. Then $\nabla:=\Phi+\nabla^H+\Phi^*$ is a connection on $E$ as well. Typical
examples are complex pre-variations of Hodge structure with a Hodge-polarization with $F_{\nabla^{\rm{Hodge}}}^{0,2}=0$. For any such complex pre-VHS there is a canonical structure of holomorphic vector bundle on each $(\{V^{r,s}\}$; it also implies that $(\nabla^{\rm Hodge})^{0,1\wedge}\theta=0$, i.e. $\theta$ is holomorphic. 

In the rest of this section we let $M$ be a compact Riemann surface, $\mu_j, j=1,\dots, n$ an assignment of singular strengths as in Section \ref{int}, and
$$S:=\{p\in M:\mu_j(p)\neq 0\text{ for some }j\}.$$ 
\begin{definition} An $n$-tuple of functions $(b_1,\dots,b_n)$ on $X:=M\setminus S$ is of type $\mu=(\mu_1,\dots,\mu_n)$ if all $b_k$ are smooth positive functions and for each $p\in S$, there exists a coordinate chart $(U, z)$ centered at $p$ and positive smooth bounded functions $\hat b_k$ on $U$ such that $$b_k=\hat b_k|z|^{\mu_k(p)}$$ on $U$, $k=1,\dots,n$.
\end{definition}

Now we are going to make some choices which will be fixed through the rest of this section:\\
\begin{itemize}
\item[(1)] In Section \ref{Cvhs} we made use of a complex line bundle $L$ satisfying (\ref{1}). We will choose a specific one as follows: fix a square-root $K_M^{\frac{1}{2}}$ of $K_M$ (which exists in any case) and take $L:=(K_M^{\frac{1}{2}})^{\otimes n}$.
\item[(2)] Fix a holomorphic atlas of charts $\{(U_\alpha,z_\alpha)\}$ on each of whose member $K_M^{\frac{1}{2}}$ is trivialized by a holomorphic frame $\sigma_\alpha$ with $\sigma_\alpha^{\otimes 2}=dz_\alpha$. If $\sigma_\beta=\psi_{\alpha\beta}\sigma_\alpha$ on $U_\alpha\cap U_\beta$, then 
\begin{equation}\label{psi-J}
\psi_{\alpha\beta}^2=J_{\alpha\beta}:=\frac{dz_\beta}{dz_\alpha}.
\end{equation}
\item[(3)] For a smooth riemannian metric $g$ on $M$, when writing 
$g=\varphi_\alpha\cdot\overline\varphi_\alpha$ with a nonvanishing $1$-form $\varphi_\alpha$ in a coordinate chart $(U_\alpha,z_\alpha)$, we always choose $\varphi_\alpha$ to be the unique positive multiple of $dz_\alpha$, i.e. $\varphi_\alpha=\lambda_\alpha dz_\alpha$, $\lambda_\alpha>0$. We have 
\begin{equation}\label{lambda-J}
\frac{\lambda_\alpha}{\lambda_\beta}=|J_{\alpha\beta}|
\end{equation} 
on $U_\alpha\cap U_\beta$. In addition, if 
$d\varphi_\alpha=-i\rho_\alpha\wedge\varphi_\alpha$ for a real $1$-form $\rho_\alpha$, then 
\begin{equation}\label{lambda-rho}
\frac{\overline\partial\lambda_\alpha}{\lambda_\alpha}=-i\rho_\alpha^{0,1}.
\end{equation}
\end{itemize}
\begin{definition}\label{ass higgs} For any assignment of singular strengths $\mu=(\mu_1,\dots,\mu_n)$, let $\widetilde V^{n-k,k}$ be vector bundles (\ref{2}) with the holomorphic structure induced by that of $K_M$ and $L$ chosen above and $\widetilde V=\oplus_k\widetilde V^{n-k,k}$.\\
\begin{itemize}
\item[(1)] Let $E=\widetilde V|_X$ and $E_k=\widetilde V^{n-k,k}|_X$. There is a natural Higgs field $\Phi$ on $V$ whose components $\Phi_k$ correspond to 
$$1\in\Gamma(M,K_M\otimes{\rm Hom}(L\otimes K_M^{\otimes-k+1},L\otimes K_M^{\otimes-k}).$$
\item[(2)] A smooth metric $H=\oplus_k H_k$ on the Higgs bundle $(V,\Phi)$ is of type $\mu$ if in a sufficiently small coordinate chart $(U,z)$ centered at any puncture $p$, $H_k$ and $|z|^{\mu_k(p)}$ are mutually bounded with respect to local trivializations of $\widetilde V^{n-k,k}$.
\end{itemize}
\end{definition}
The key correspondence is given by the following proposition.
\begin{proposition}\label{main bij} 
There is a one-one correspondence between $n$-tuple of functions $(b_1,\dots,b_n)$ on $X=M\setminus S$ of type $\mu$ and hermitian metrics $H$ of type $\mu$ on the Higgs bundle $(V,\Phi)$ with $\det H=1$. Under this correspondence, a solution $(b_1,\dots,b_n)$ to (\ref{11}) corresponds to a metric $H$ with $\nabla=\Phi+\nabla^H+\Phi^*$ flat.
\end{proposition}
\begin{proof}
We start with $(b_1,\dots,b_n)$ of type $\mu$ first. By Proposition \ref{d-pre-vhs}, we have the unique diagonal complex pre-VHS $(V:=\oplus_k V^{n-k,k},\nabla)$ (where $V^{n-k,k}$ is the restriction of $\widetilde V^{n-k,k}$) with pre-Hodge field $(b_1,\dots,b_n)$ polarized by the metric $h$ naturally associated to $g$. As mentioned at the beginning of this section, this complex pre-VHS determines a Higgs bundle, which will seen to be isomorphic to the Higgs bundle $E$ in Definition \ref{ass higgs}, which has $\widetilde V$ as a natural holomorphic extension. 

In each coordinate chart $(U_\alpha,z_\alpha)$ we select the unitary frame
$(e_k)_\alpha:=(\sigma_\alpha/|\sigma_\alpha|)^{\otimes n}\otimes\varphi_\alpha^{\otimes-k}$ for $\widetilde V^{n-k,k}$, $k=1,\dots,n$. We need to get a holomorphic frame $(e'_k)_\alpha=(\delta_k)_\alpha(e_k)_\alpha$ on $U_\alpha\setminus S$ for each $\alpha$. This means $(\nabla^{\rm Hodge})^{0,1}(e'_k)_\alpha=0$, or equivalently (by (\ref {13}),(\ref{14}), and (\ref{lambda-rho}))
$$\frac{\overline\partial(\delta_k)_\alpha}{(\delta_k)_\alpha}+\frac{\overline\partial[(b_1^nb_2^{n-1}\cdots b_n)^{\frac{1}{n+1}}(b_1\cdots b_k)^{-1}]}{(b_1^nb_2^{n-1}\cdots b_n)^{\frac{1}{n+1}}(b_1\cdots b_k)^{-1}}+\frac{\overline\partial\lambda^{\frac{n}{2}-k}_\alpha}{\lambda^{\frac{n}{2}-k}_\alpha}=0.$$ 
Therefore, we may take $(\delta_k)_\alpha$ to be $(b_1^nb_2^{n-1}\cdots b_n)^{\frac{-1}{n+1}}b_1\cdots b_k\lambda_\alpha^{k-\frac{n}{2}}$ multiplied by any holomorphic function on $U_\alpha\setminus S$. The simplest choice is

\begin{equation}\label{delta}
(\delta_k)_\alpha:=(b_1^nb_2^{n-1}\cdots b_n)^{\frac{-1}{n+1}}b_1\cdots b_k\lambda_\alpha^{k-\frac{n}{2}}.
\end{equation} 
\\
Then $(e'_k)_\beta=\psi_{\alpha\beta}^nJ_{\alpha\beta}^{-s}(e'_k)_\alpha$:

\begin{eqnarray*}
\frac{(e'_k)_\beta}{(e'_k)_\alpha}&=&\left(\frac{\lambda_\beta}{\lambda_\alpha}\right)^{k-\frac{n}{2}}\left(\frac{\sigma_\beta}{\sigma_\alpha}\right)^n\left|\frac{\sigma_\beta}{\sigma_\alpha}\right|^{-n}\left(\frac{\varphi_\beta}{\varphi_\alpha}\right)^{-k}\\
\\
&=&\psi_{\alpha\beta}^nJ_{\alpha\beta}^{-k}.
\end{eqnarray*}
\\
This shows that $E_k$ is the restriction of $\widetilde V^{n-k,k}$. We define a smooth metric $H_k$ on $E_k$ by setting on $U_\alpha\setminus S$ 
\begin{equation}\label{h-h}
(H_k)_\alpha:=|(\delta_k)_\alpha|^2.
\end{equation}
It is direct to show that this defines
a metric on $E_k$, which is of type $\mu$ since 
$$(b_1^nb_2^{n-1}\cdots b_n)^{\frac{-1}{n+1}}b_1\cdots b_k|(f_k)_\alpha|^{-1}$$ 
has the same order as $|z_\alpha^{d_k(p)}|$ near the punctures $p\in S$ by the assumption that $(b_1,\dots,b_n)$ is of type $\mu$.

Conversely, from a metric of bounded type $H=\oplus_k H_k$ 
we can get 
$$(b_1^nb_2^{n-1}\cdots b_n)^{\frac{-1}{n+1}}b_1\cdots b_k$$ 
back from (\ref{h-h}) and (\ref{delta}) and hence can obtain all $b_k$. This establishes the expected one-one correspondence. 

As for the last statement, note that the underlying bundles on both sides of the correspondence are smoothly equivalent over the compliment $X$ the punctures $S$ and the connections on both sides are equivalent under this identification. By Proposition \ref{toda-cvhs1} the proof is completed.
\end{proof}

By Proposition \ref{main bij}, finding solutions to Toda systems of VHS type with singular strength assignment $\mu$ becomes finding hermitian metrics $H$ of type $\mu$ on $V$ with $\nabla=\Phi+\nabla^{H}+\Phi^*$. In the next section we will obtain a criterion of existence of such kind of metrics.
\section{Stability and Higgs-Hermitian-Yang-Mills metrics}\label{simpson}
In this section we provide the criterion for the existence of solutions to Toda systems of VHS type and prove their uniqueness. Our main ingredient is Simpson's theory of constructing Higgs-Hermitian-Yang-Mills metrics on Higgs bundles from stability \cite{simpson}. Let $(X,g)$ be a Ka\"hler manifold, $(E,\Phi)$ a Higgs bundle over $X$ and $H$ a smooth hermitian metric on $E$ as in Section \ref{Higgs}.
Recall that we have set $$\nabla:=\Phi+\nabla^H+\Phi^*.$$
In the following, we will denote the curvature $F_\nabla$ of such $\nabla$ induced by $H$ as $F^{\rm Higgs}_H$.
We suppose that $(X,g)$ satisfies suitable assumptions (cf. \cite{simpson}, Section 2, Assumptions 1,2 and 3), which will be fulfilled in our situation (cf. \cite{simpson}, Propositons 2.2 or 2.4). 
\begin{definition}\label{hym}
The metric $H$ is a Hermitian-Yang-Mills metric if 
the trace-free part of $\Lambda_gF^{\rm Higgs}_H$ vanishes.
\end{definition}
We will need the notion of stability. Let $A$ be a group acting by biholomorphic maps of $X$ preserving the metric $g$ and acting compatibly by automorphisms $a:E\longrightarrow E$ preserving the metric $K$ and acting on $\Phi$ by homotheties $a\Phi a^{-1}=\lambda (a)\Phi .$
\begin{definition}\label{mis} (\cite{simpson}, p.$877,878$)
\begin{itemize}
\item[(1)] A sub-Higgs sheaf of a Higgs bundle $(E,\Phi)$ is an analytic subsheaf $\mathcal V\subset\mathcal O(E)$ such that 
$\Phi:\mathcal V\longrightarrow\mathcal O(K_X)\otimes\mathcal V$. (If $\mathcal V$ is saturated, outside a set of codimension $2$ it is the coherent sheaf associated to a subbundle of $E$.)
\item[(2)] For a saturated subsheaf $\mathcal V$ and a smooth metric $K$ on $E$ such that ${\sup_X|\Lambda_gF^{\rm Higgs}_K|_K}<\infty$, 
$${\rm deg}(\mathcal V,K):=i\int_X{\rm Tr}\ \Lambda_gF^{\rm Higgs}_K.$$
(This is either a real number of $-\infty$ by \cite{simpson}, Lemma 3.2.)
\item[(3)] $(E,\Phi,K)$ is stable with respect to the $A$-action if for every proper saturated sub-Higgs sheaf $\mathcal V$ preserved by $A$,
$$\frac{{\rm deg}(\mathcal V,K)}{{\rm rk}(\mathcal V)}<\frac{{\rm deg}(E,K)}{{\rm rk}(E)}.$$ 
\end{itemize}
\end{definition}
Our main tool is the following result due to Simpson (cf. \cite{simpson}, Theorem 1 and Proposition 3.3).
\begin{theorem}\label{CS} $(\bf Simpson)$ Let $(X,g)$, $(E,\Phi)$, $K$, and $A$ satisfy all conditions above. If $(E,\Phi,K)$ is stable with respect to the $A$-action, then there exists a smooth $A$-invariant Higgs-Hermitian-Yang-Mills metric $H$ with $H$ and $K$ mutually bounded, 
$$\det H=\det K\text{, and }\ \overline\partial h+[\Phi,h]\in L^2_{g,K},$$
where $h$ is the unique endomorphism of $E$ such that
$$(\cdot,\cdot)_H=(h(\cdot),\cdot)_K.$$
If furthermore $\Phi\wedge\Phi=0$, the first Chern form $c_1(E,K)=0$, and 
$\int_Xc_2(E,K)\wedge\omega_g^{n-2}=0$, then the connection $\nabla=\Phi+\nabla^H+\Phi^*$ is flat.  Conversely, if there exists an $A$-invariant Higgs-Hermitian-Yang-Mills metric, then
$$\frac{{\rm deg}(\mathcal V,K)}{{\rm rk}(\mathcal V)}\leq\frac{{\rm deg}(E,K)}{{\rm rk}(E)}$$ for every proper saturated sub-Higgs sheaf $\mathcal V$ preserved by $A$ and equality holds only if $E=\mathcal V\oplus \mathcal V^{\perp}$ is an orthogonal direct sum of Higgs subbundles.
\end{theorem}
Now we give the proof of our main result.
\begin{theorem}\label{main thm} Let $M$ be a compact Riemann surface with a smooth riemannian metric $g$. Let $\mu=(\mu_1,\dots,\mu_n)$ be an assignment of singular strengths and 
$$S:=\{p\in M:\mu_j(p)\neq 0\text{ for some }j\}.$$ There exists at most one $\mu$-admissible solution $(u_1,\dots,u_n)$ to the Toda system of VHS type
$$
\begin{pmatrix}
\frac{1}{4}\Delta_gu_1-\frac{K_g}{2}\\
\frac{1}{4}\Delta_gu_2-\frac{K_g}{2}\\
\vdots\\
\frac{1}{4}\Delta_gu_n-\frac{K_g}{2}
\end{pmatrix}
=
\begin{pmatrix}
2&-1&&\\
-1&2&&\\
&&\ddots&-1\\
&&-1&2
\end{pmatrix}
\begin{pmatrix}
e^{u_1}\\
e^{u_2}\\
\vdots\\
e^{u_n}
\end{pmatrix}.
$$
There exists a $\mu$-admissible solution if and only if
\begin{equation}\label{condition}
d_{n-l+1}+\cdots +d_n <l(n-l+1)({\rm genus}(M)-1),
\end{equation}
$l=1,\dots,n,$
where 
$$d_k:=\sum_{p\in M}\left(-\frac{1}{n+1}\sum_{j=1}^n(n-j)\mu_j(p)+\sum_{j=1}^k\mu_j(p)\right),$$ 
$k=1,\dots,n$.
\end{theorem}
\begin{proof} We prove the statement about uniqueness first. The argument is essentially the same as that in the proof of Lemma 10.9 of \cite{simpson}. Let $(E,\Phi)$ be as in Definition \ref{ass higgs} (1). Let $A$ be the group $U(1)^{\times (n+1)}\cap SU(n+1)$ acting on $X$ trivially and on $E$ in the obvious diagonal manner. Suppose we have two solutions corresponding to Higgs-Hermitian-Yang-Mills metrics $H$ and $H'$ respectively. Let $h$ be the unique endomorphism of $E$ such that $(\cdot,\cdot)_{H'}=(h(\cdot),\cdot)_H$. Note that $h$ is a positive definite self-adjoint and bounded with respect to $H$. Taking trace of Lemma 3.1 $(c)$ in \cite{simpson} gives 
$$\Delta_d{\rm Tr}\ h=2\Delta_\partial{\rm Tr}\ h=-\left|\big(\overline\partial h+[\Phi,h]\big)h^{\frac{1}{2}}\right|^2_{H}\leq 0.$$
As mentioned above, Assumption 3 in \cite{simpson} holds for $(X,g|_X)$, and hence a positive bounded subharmonic function must be harmonic. Therefore, $$\overline\partial h+[\Phi,h]=0,$$ 
which is equivalent to saying that $h$ is a holomorphic endomorphism of $E$ commuting with $\Phi$. Since $h$ commutes with the $A$-action, it acts on $E$ diagonally by multiplication with positive numbers $h_0,\dots,h_n$. The commutativity of $h$ with $\Phi$ implies $h_0=\cdots=h_n$. Finally, $h_0\cdots h_n=1$ since $\det H'=\det H=1$. This shows that $h={\rm id_E}$ and the proof is completed. 

Now let $(E, \Phi)$ and $E_k$ be as in Definition \ref{ass higgs}. Suppose that $A$ acts on $E$ diagonally. We equip $E$ with a metric $K:=\oplus_jK_j$ of type $\mu$ (Definition \ref{ass higgs}) such that $\det K=1$. Such a metric can be constructed as follows. Choose a covering of $M$ by coordinate disks $(U_\alpha,z_\alpha)$ and a partition of unity $\{\rho_\alpha\}$ subordinate to $\{U_\alpha\}$. We may assume that 
\begin{itemize}
\item[(1)] each puncture $s\in S$ is contained in $U_\alpha$ for exactly one $\alpha$, denoted as $\alpha(s)$, and $z_{\alpha(s)}(s)=0$;
\item[(2)] for every $s\in S$, there exists an open 
neighborhood $W_{\alpha(s)}\subset U_{\alpha(s)}$ of $s$ such that $\rho_{\alpha(s)}|_{W_{\alpha(s)}}\equiv 1$;
\item[(3)] every $\widetilde V^{n-j,j}$ is trivialized on $U_\alpha$ by a holomorphic local frame $e_{j,\alpha}$.
\end{itemize}
Let $\sigma$ be an element in the fibre of $E_j$ above $p\in M\setminus S$. We write $$\sigma=\sigma_\alpha e_{j,\alpha}(p)$$ 
if $p\in U_\alpha$. Let 
$$f_{\alpha}:=\left\{
\begin{array}{cl}
|z_\alpha|^{2d_k(s)},&\text{ if }s\in U_\alpha\cap S;\\
\\
1,&\text{ otherwise.}
\end{array}
\right.$$
We define the metric $K_j$ on $E_j$ for $j=1,\dots,n$ by setting
$$|\sigma|^2_{K_j}:=\sum_{\alpha:p\in U_\alpha}\rho_\alpha(p)f_\alpha(p)|\sigma_{\alpha}|^2.$$
Finally we let $K_0$ on $E_0$ be defined so that $\det K=1$.

$c_1(E,K)=0$ by the construction of $K$.  $\Phi\wedge\Phi=0$ and $c_2(E,K)=0$ automatically on a Riemann surface.

Since ${\rm dim}\ M=1$, proper saturated sub-Higgs sheaves are exactly proper holomorphic subbundles of $E$ preserved by $\Phi$. By the form of $\Phi$ (a nilpotent string), it is clear that such kind of subbundles are exactly 
$$F^l:=E_{n-l+1}\oplus\cdots\oplus E_n,$$
$ l=1,\dots,n$. 
By Definition \ref{mis},
$${\rm deg}(F^l,K)=\sum_{j=n-l+1}^n{\rm deg}(E_{j},K_j).$$

Note that for each $j$ the smooth metric $K_j$ on $E_j$ can be viewed as a singular metric on $\widetilde V^{n-j,j}$, whose curvature current represents the first Chern class of $\widetilde V^{n-j,j}=L\otimes K_M^{\otimes -j}$.
By the Poincar{\'e}--Lelong formula,
$${\rm deg}{\widetilde V^{n-j,j}}=
{\rm deg}(E_j,K_j)-d_j.$$
Therefore
$${\rm deg}(E_j,K_j)=({\rm genus}(M)-1)(n-2j)+d_j,$$
and hence $${\rm deg}(F^l,K)=({\rm genus}(M)-1)l(l-n-1)+d_{n-l+1}+\cdots+d_n.$$
It is clear that ${\rm deg}(E,K)=0$. If (\ref{condition}) holds for $l=1,\dots,n$, then $(E,\Phi,K)$ is stable. In view of Proposition \ref{main bij} and the first part of Theorem \ref{CS}, we obtain a solution to the Toda system with required asymptotic behaviour near punctures. Conversely, if there exists such a solution, by Proposition \ref{main bij} and the second part of Theorem \ref{CS}, we have 
$$d_{n-l+1}+\cdots+d_n\leq l(n-l+1)({\rm genus}(M)-1),$$
$l=1,\dots,n$. Actually, all inequalities above are strict since the orthogonal complement of $F^l$ is not a sub-Higgs sheaf, $l=1,\dots,n$. Therefore (\ref{condition}) holds for $l=1,\dots,n$.
\end{proof}

\bigskip

\begin{thebibliography}{99}
\bibitem{baraglia} D. Baraglia; 
{\it $G_2$ Geometry and integrable systems}, Ph. D. thesis, Trinity (2009).\\

\bibitem{hitchin} N. Hitchin;
{\it The self-duality equations on a Riemann surface}, Proc. London Math. Soc. (3) 55 (1987), no. 1, 59--126.\\

\bibitem{jost-wang} J. Jost and G. Wang;
{\it Classification of solutions of a Toda system in $R^2$}, Int. Math. Res. Not. 2002, no. 6, 277--290.\\

\bibitem{lin-wang} C.-S. Lin and C.-L. Wang;
{\it Elliptic functions, Green functions and the mean field equations on tori}, Ann. of Math. (2) 172 (2010), no. 2, 911--954.\\

\bibitem{simpson} C. Simpson;
{\it Constructing variations of Hodge structure using Yang-Mills theory and applications to uniformization}, J. Amer. Math. Soc. 1 (1988), no. 4, 867--918.

\end{thebibliography}
\end{document}